\documentclass[12pt,a4paper]{amsart}
\usepackage{amsfonts}
\usepackage{amsthm}
\usepackage{amsmath}
\usepackage{amscd}
\usepackage[latin2]{inputenc}
\usepackage{t1enc}
\usepackage[mathscr]{eucal}
\usepackage{indentfirst}
\usepackage{graphicx}
\usepackage{graphics}
\usepackage{pict2e}
\usepackage{epic}
\numberwithin{equation}{section}
\usepackage[margin=2.9cm]{geometry}
\usepackage{epstopdf} 
\usepackage{mathtools}
\usepackage{tikz}
\usetikzlibrary{positioning}
\usepackage{url}
\usepackage{hyperref}
 
\usepackage{amssymb}
\usepackage{array}
\usepackage{wrapfig}
\usepackage{multirow}
\usepackage{tabu}

\theoremstyle{plain}
\newtheorem{Th}{Theorem}[section]
\newtheorem{Lemma}[Th]{Lemma}

\newtheorem{Prop}[Th]{Proposition}

 \theoremstyle{definition}
\newtheorem{Def}[Th]{Definition}
\newtheorem{Conj}[Th]{Conjecture}
\newtheorem{Rem}[Th]{Remark}
\newtheorem{?}[Th]{Problem}
\newtheorem{Ex}[Th]{Example}

\begin{document}

\title{A combinatorial formula for the Ehrhart $h^*$-vector of the hypersimplex}

\author{Donghyun Kim}

\email{donghyun\_kim@berkeley.edu}



\begin{abstract} We give a combinatorial formula for the Ehrhart $h^*$-vector of the hypersimplex. In particular, we show that $h^{*}_{d}(\Delta_{k,n})$ is the number of hypersimplicial decorated ordered set partitions of type $(k,n)$ with winding number $d$, thereby proving a conjecture of N. Early. We do this by proving a more general conjecture of N. Early on the Ehrhart $h^*$-vector of a generic cross-section of a hypercube.   
\end{abstract}

\maketitle

\section{Introduction} For two integers $0<k<n$, the $(k,n)$\textit{-th hypersimplex} is defined to be 
$$\Delta_{k,n}=\{(x_1,\cdots,x_n) \in \mathbb{R}^n \mid 0 \leq x_i \leq 1,   x_1+\cdots+x_n=k\}.$$

It is an $(n-1)$-dimensional polytope inside $\mathbb{R}^n$ whose vertices are (0,1)-vectors with exactly $k$ 1's. In particular it is an integral polytope. The hypersimplex can be found in several algebraic and geometric contexts, for example, as a moment polytope for the torus action on the Grassmannian, or as a weight polytope for the fundamental representation of $GL_n$.

For an $n$-dimensional integral polytope $\mathcal{P} \subset \mathbb{R}^N$, it is well known from Ehrhart theory that the map $r \rightarrow |r\mathcal{P} \cap \mathbb{Z}^N|$ is a polynomial function in $r$ of degree $n$, which we call \textit{Ehrhart polynomial}, and corresponding \textit{Ehrhart series} $\sum_{r=0}^{\infty} |r\mathcal{P} \cap \mathbb{Z}^N|t^r$ is a rational function of the form
$$\sum_{r=0}^{\infty} |r\mathcal{P} \cap \mathbb{Z}^N|t^r=\frac{h^{*}(t)}{(1-t)^{n+1}},$$
such that $h^{*}(t)$ is a polynomial of degree $\leq n$ (see \cite{EC1}). Define $h^{*}_d$ to be the coefficient of $t^d$ in $h^{*}(t)$. The vector $(h^{*}_0,\cdots,h^{*}_{n})$ is called the Ehrhart $h^{*}$-vector of $\mathcal{P}$ and $h^{*}(t)$ is called the $h^{*}$-polynomial of $\mathcal{P}$. A standard result from Ehrhart theory is that $\sum\limits_{i=0}^{n} h^{*}_i$ equals the normalized volume of $\mathcal{P}$. 

For a permutation $w \in S_n$, we say $i \in [n-1]$ is a $descent$ of $w$ if $w(i)>w(i+1)$ and define $des(w)$ to be the number of descents of $w$. The $Eulerian$ number $A_{k,n-1}$ is the number of $w \in S_{n-1}$ with $des(w)=k-1$. A well-known fact about the hypersimplex $\Delta_{k,n}$ is that its normalized volume is $A_{k,n-1}$ (see \cite{Stan}). So we have 
$$\sum_{d=0}^{n-1} h^{*}_d(\Delta_{k,n})=A_{k,n-1}.$$

In general, the entries of the $h^{*}$-vector of an integral polytope are nonnegative integers (see \cite{Sta}). It has been an open problem for some time to give a combinatorial interpretation of $h^{*}_{d}(\Delta_{k,n})$. In \cite{Li}, N. Li gave a combinatorial interpretation of $h^{*}_{d}(\Delta^{'}_{k,n})$, where $\Delta^{'}_{k,n}$ is the hypersimplex with the lowest facet removed, using permutations $w \in S_{n-1}$ and their descents, excedances, and covers. In \cite{Early1}, N. Early conjectured a combinatorial interpretation for $h^{*}_{d}(\Delta_{k,n})$ using hypersimplicial decorated ordered set partitions of type $(k,n)$.  

In \cite{Katz}, Katzman computed the Hilbert series of algebras of Veronese type, which gives a formula for the Ehrhart series of the hypersimplex $\Delta_{k,n}$ as a special case. The formula is
   \begin{equation} \label{1} \frac{\sum\limits_{i\geq0} (-1)^{i}\binom{n}{i} \left({(\sum\limits_{j \geq 0} \binom{i}{j} (t-1)^j (\sum\limits_{l \geq 0} \binom{n-j}{l(k-i)}_{\hspace{-1mm}k-i} t^l)}\right) }{(1-t)^n} 
   \end{equation}
where the notation $\binom{n}{b}_a$ means the coefficient of $t^b$ in $(1+t+\cdots+t^{a-1})^n$. For example, when $a=2$, it becomes an ordinary binomial coefficient. The numerator of (\ref{1}) is the $h^{*}$-polynomial of the hypersimplex $\Delta_{k,n}$, thus giving an explicit formula for its $h^{*}$-vector. However, it doesn't give a combinatorial or manifestly positive formula for the $h^{*}$-vector.  

In this paper, we prove N. Early's conjecture by relating it to (\ref{1}). We now explain  the conjecture. A \textit{decorated ordered set partition} $((L_1)_{l_1},\cdots,(L_m)_{l_m})$ of type $(k,n)$ consists of an ordered partition $(L_1,\cdots,L_m)$ of $\{1,2,...,n\}$ and an $m$-tuple $(l_1,\cdots,l_m) \in \mathbb{Z}^m$ such that $l_1+\cdots+l_m=k$ and $l_i \geq 1$. We call each $L_i$ a $block$ and we place them on a circle in a clockwise fashion then think of $l_i$ as the clockwise distance between adjacent blocks $L_i$ and $L_{i+1}$ (indices are considered modulo $m$). So the circumference of the circle is $l_1+\cdots+l_m=k$. We regard decorated ordered set partitions up to cyclic rotation of blocks (together with corresponding $l_i$). For example, decorated ordered set partition $(\{1,2,7\}_2,\{3,5\}_3,\{4,6\}_1)$ is same as $(\{3,5\}_3,\{4,6\}_1,\{1,2,7\}_2)$. A decorated ordered set partition is called \textit{hypersimplicial} if it satisfies $1 \leq l_i \leq |L_i|-1$ for all $i$. For the motivation and more background on decorated ordered set partitions, see \cite{Early2}.

\begin{Ex} \label{ex1}
Consider a decorated ordered set partition $(\{1,2,7\}_2,\{3,5\}_3,\{4,6\}_1)$ of type (6,7) (see Figure \ref{fig1}). This is not hypersimplicial as $3>|\{3,5\}|-1$.

By inserting empty spots, we can encode the distance information. For example, the (clockwise) distance between $\{1,2,7\}$ and $\{3,5\}$ is 2 so we insert one empty spot on the circle between those blocks. The distance between $\{3,5\}$ and $\{4,6\}$ is 3 so we insert two empty spots. We obtain the figure on the right as a result. Including empty spots, there will be $k=6$ spots total. 

\end{Ex}
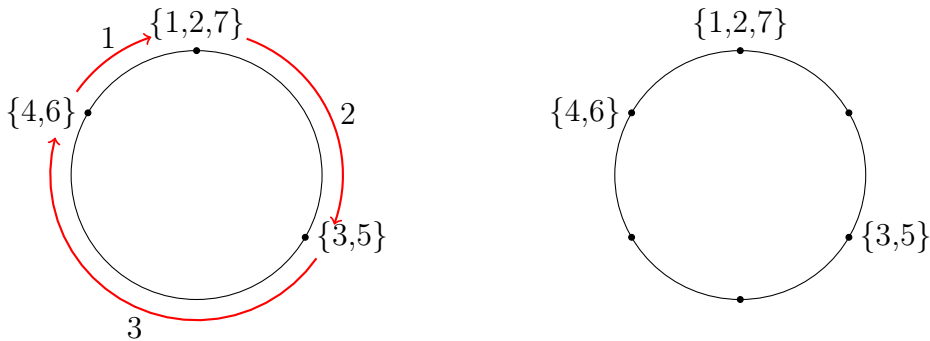
\begin{figure}[ht] 
\begin{tikzpicture}[scale=0.55]
\filldraw[black] (0,3) circle (2pt) node[anchor=south] {\{1,2,7\}};
\filldraw[black] (2.598,-1.5) circle (2pt) node[anchor=west] {\{3,5\}};
\filldraw[black] (-2.598,1.5) circle (2pt) node[anchor=east] {\{4,6\}};
\draw (0,0) circle (3);

\filldraw[black] (5+8,3) circle (2pt) node[anchor=south] {\{1,2,7\}};
\filldraw[black] (5+10.598,1.5) circle (2pt) node[anchor=west] {};
\filldraw[black] (5+10.598,-1.5) circle (2pt) node[anchor=west] {\{3,5\}};
\filldraw[black] (5+8,-3) circle (2pt) node[anchor=north] {};
\filldraw[black] (5+8-2.598,-1.5) circle (2pt) node[anchor=east] {};
\filldraw[black] (5+8-2.598,1.5) circle (2pt) node[anchor=east] {\{4,6\}};
\draw (5+8,0) circle (3);

\draw [red,thick,<-,domain=-20:70] plot ({3.5*cos(\x)}, {3.5*sin(\x)});
\draw [red,thick,->,domain=145:108] plot ({3.5*cos(\x)}, {3.5*sin(\x)});
\draw [red,thick,<-,domain=165:325] plot ({3.5*cos(\x)}, {3.5*sin(\x)});
\filldraw[black] (3.672-0.5,9-7.520) circle (0.00001pt) node[anchor=west] {2};
\filldraw[black] (-8.5+7.0208360839,9-12.172077255
) circle (0.00001pt) node[anchor=north] {3};
\filldraw[black] (-8.5+6.393647419
,9-6.2047757148
) circle (0.00001pt) node[anchor=south] {1};

\end{tikzpicture}
\caption{The figure on the left is the picture associated to the decorated ordered set partition $(\{1,2,7\}_2,\{3,5\}_3,\{4,6\}_1)$. The figure on the right is the picture obtained after inserting empty spots.} \label{fig1}
\end{figure}

Given a decorated ordered set partition, we define the \textit{winding vector} and the \textit{winding number}. To define the winding vector, let $w_i$ be the distance of the path starting from the block containing $i$ to the block containing $(i+1)$ moving clockwise (where $i$ and $(i+1)$ are considered modulo $n$). If $i$ and $(i+1)$ are in the same block then $w_i=0$. In Figure \ref{fig1}, the winding vector is $w=(0,2,3,3,3,1,0)$.

The total length of the path is $(w_1+\cdots+w_n)$, which should be a multiple of $k$ as we started from 1 and came back to 1 moving clockwise. If $(w_1+\cdots+w_n)=kd$, then we define the winding number to be $d$. In Figure \ref{fig1}, the winding number is 2. 

\begin{Rem}It is known that hypersimplicial decorated ordered set partitions of type $(k,n)$ are in bijection with $w \in S_{n-1}$ such that $des(w)=k-1$ (see \cite{Ocn}). 
\end{Rem}
Now we will state the conjectures of N. Early. 

\begin{Conj} [\cite{Early1}, Conjecture 1] \label{conj1}
The number of hypersimplicial decorated ordered set partitions of type $(k,n)$ with winding number $d$ is $h^{*}_d(\Delta_{k,n})$. 
\end{Conj}

Next we will state a more general version of Conjecture \ref{conj1} for a generic cross section of a hypercube. 

\begin{Def} For positive integers $r,k,$ and $n$, the \textit{generic cross section of a hypercube} is
$$I_{r,k}^{n}=\{(x_1,\cdots,x_n) \in [0,r]^n \mid \sum_{i=1}^{n} x_i =k \}.$$
\end{Def}

When $r=1$, it is the hypersimplex $\Delta_{k,n}$.

\begin{Def}
A decorated ordered set partition $P=((L_1)_{l_1},\cdots,(L_m)_{l_m})$ is $r$-$hypersimplicial$ if $ 1 \leq l_i \leq r|L_i|-1$ for all $i$.
\end{Def}

Note that the notions of hypersimplicial and 1-hypersimplicial are equivalent. The decorated ordered set partition $(\{1,2,7\}_2,\{3,5\}_3,\{4,6\}_1)$ in Example \ref{ex1} is not hypersimplicial, but it is $r$-hypersimplicial for $r \geq 2$.

\begin{Conj} [\cite{Early1}, Conjecture 6] \label{conj2}
The number of $r$-hypersimplicial decorated ordered set partitions of type $(k,n)$ with winding number $d$ is $h_{d}^{*}(I_{r,k}^{n})$.
\end{Conj}

Our goal is to prove Conjecture \ref{conj2} and derive Conjecture \ref{conj1} as specializing to $r=1$. 
\section{Proof of Conjecture \ref{conj2}}
\subsection{A simplification of Katzman's formula}

Again using the formula for Hilbert series of algebras of Veronese type (see \cite{Katz}), the Ehrhart series of $I_{r,k}^{n}$ is 
\begin{equation} \label{2} \frac{\sum\limits_{i \geq 0} (-1)^{i}\binom{n}{i} \left(\sum\limits_{j\geq0} \binom{i}{j} (t-1)^j (\sum\limits_{l \geq 0} \binom{n-j}{l(k-ri)}_{\hspace{-1mm}k-ri} t^l)\right)}{(1-t)^n}. 
   \end{equation}

Now we simplify (\ref{2}) to get a simple description for the $h^{*}$-vector of $I_{r,k}^{n}$.

\begin{Lemma} \label{lem1}
For positive integers $n,m$,and $a$, we have $$\binom{n}{m}_{a} -\binom{n}{m-1}_{a}=\binom{n-1}{m}_{a} -\binom{n-1}{m-a}_{a}.$$
\end{Lemma}
\begin{proof}
By a combinatorial argument, we have $\binom{n}{m}_a=\sum\limits_{k=0}^{a-1}\binom{n-1}{m-k}_a$ and $\binom{n}{m-1}_a=\sum\limits_{k=0}^{a-1}\binom{n-1}{m-1-k}_a$. Subtracting these two gives the lemma. 
\end{proof}

\begin{Prop} \label{prop1} For positive integers $s$ and $a$, we have
$$\sum\limits_{j \geq 0}\binom{s}{j}(t-1)^j(\sum\limits_{l \geq 0} \binom{n-j}{la}_{\hspace{-1.3mm}a} t^l)=\sum\limits_{l \geq 0} \binom{n}{la-s}_{\hspace{-1.3mm}a} t^l.$$
\end{Prop}
\begin{proof}
We proceed by induction on $s$. For $s=0$, this is a trivial identity. Let's assume that the proposition holds for $s=u-1$ and for all $n$, which means 
\begin{equation}\label{3}
\sum\limits_{j \geq 0}\binom{u-1}{j}(t-1)^j(\sum\limits_{l \geq 0} \binom{n-j}{la}_{\hspace{-1.3mm}a} t^l)=\sum\limits_{l \geq 0} \binom{n}{la-u+1}_{\hspace{-1.3mm}a} t^l.
\end{equation}

Now replacing $n$ with $(n-1)$ and multiplying by $(t-1)$ we have

$$\sum\limits_{j\geq0}\binom{u-1}{j}(t-1)^{j+1}(\sum\limits_{l \geq 0} \binom{n-1-j}{la}_{\hspace{-1.3mm}a} t^l)=\sum\limits_{l \geq 0} \binom{n-1}{la-u+1}_{\hspace{-1.3mm}a} t^l(t-1).$$

Replacing $j$ with $(j-1)$ and rearranging the righthand side gives 

\begin{equation}
    \label{44}
    \sum\limits_{j\geq0}\binom{u-1}{j-1}(t-1)^{j}(\sum\limits_{l \geq 0} \binom{n-j}{la}_{\hspace{-1.3mm}a} t^l)=\sum\limits_{l \geq 0} (\binom{n-1}{(l-1)a-u+1}_{\hspace{-1.3mm}a}-\binom{n-1}{la-u+1}_{\hspace{-1.3mm}a}) t^l.
\end{equation}

Summing (\ref{3}) and (\ref{44}), and using Lemma \ref{lem1} gives
$$\sum\limits_{j \geq 0}\binom{u}{j}(t-1)^j(\sum\limits_{l \geq 0} \binom{n-j}{la}_{\hspace{-1.3mm}a} t^l)=\sum\limits_{l \geq 0} \binom{n}{la-u}_{\hspace{-1.3mm}a} t^l.$$

\end{proof}

Using Proposition \ref{prop1}, the Ehrhart series of $I_{r,k}^{n}$ (\ref{2}) becomes

$$ \frac{\sum\limits_{i \geq 0} (-1)^{i}\binom{n}{i} \left(\sum\limits_{l\geq0} \binom{n}{l(k-ri)-i)}_{k-ri}t^l \right)} {(1-t)^n}.
   $$
Thus we have
\begin{equation}
    \label{5}
    h^{*}_d(I_{r,k}^{n})=\sum\limits_{i \geq 0} (-1)^i \binom{n}{i} \binom{n}{(k-ri)d-i}_{k-ri}.
\end{equation}

In Section 2.2, we will prove Conjecture \ref{conj2} which contains Conjecture \ref{conj1} as a special case when $r=1$. Since we have an explicit formula for $h^{*}_d(I_{r,k}^{n})$, our strategy is to count the number of $r$-hypersimplicial decorated ordered set partitions of type $(k,n)$ with winding number $d$ and compare the formulas.

\subsection{Enumeration of \texorpdfstring{$r$-} hhypersimplicial decorated ordered set partitions with a fixed winding number} We start with an elementary lemma, skipping the proof. 
\begin{Lemma} \label{lem2}
The $\mathbb{Z}/n\mathbb{Z}$ action on $\{1,2,\cdots,n\}$ by cyclic shift does not change the \textit{winding number} of decorated ordered set partitions.
\end{Lemma}
For example, decorated ordered set partitions $(\{1,2,7\}_2,\{3,5\}_3,\{4,6\}_1)$ and\\ 
$(\{2,3,1\}_2,\{4,6\}_3,\{5,7\}_1)$ have the same winding number. 

Next we will show that a winding vector determines a decorated ordered set partition. We observed that when the winding number is $d$, then $w_1+\cdots+w_n=kd$. And $0\leq w_i \leq k-1$ since the circumference of the circle is $k$ (if $w_i=k$, then $i$ and $(i+1)$ are in a same block which means $w_i=0$). It turns out that these are the only restrictions for winding vectors.
\begin{Prop} \label{prop2}
Decorated ordered set partitions of type $(k,n)$ with winding number $d$ are in bijection with elements of $\{(w_1,\cdots,w_n) \in \mathbb{Z}^n \mid 0 \leq w_i \leq k-1 , \hspace{2mm} w_1+\cdots+w_n=kd\}$.
\end{Prop}

\begin{proof}
It is enough to construct a decorated ordered set partition of type $(k,n)$ with winding number $d$ from a winding vector satisfying the above conditions. First, draw $k$ spots on the circle in clockwise order and put 1 in one spot. Having put $i$ in some spot, move clockwise $w_i$ spots and put $i+1$ in that spot. After placing all elements, nonempty spots become blocks and the clockwise distance from $L_i$ and $L_{i+1}$ is $l_i$.
\end{proof}

\begin{Ex}
For type $(k,n)=(6,7)$, we will construct a decorated ordered set partition from the vector $(0,2,3,3,3,1,0)$. See Figure \ref{fig2}. First, draw $k=6$ spots and put 1 in one spot (upper-left figure). Then put elements according to the given vector (upper-right figure). $\{1,2,7\}$, $\{3,5\}$, and $\{4,6\}$ will be blocks. There is one empty spot between $\{1,2,7\}$ and $\{3,5\}$ so the distance is 2. The distance between $\{3,5\}$ and $\{4,6\}$ is 3 as there are two empty spots. Resulting decorated ordered set partition is $(\{1,2,7\}_2,\{3,5\}_3,\{4,6\}_1)$ (lower figure). We recovered Example \ref{ex1}.
\end{Ex}

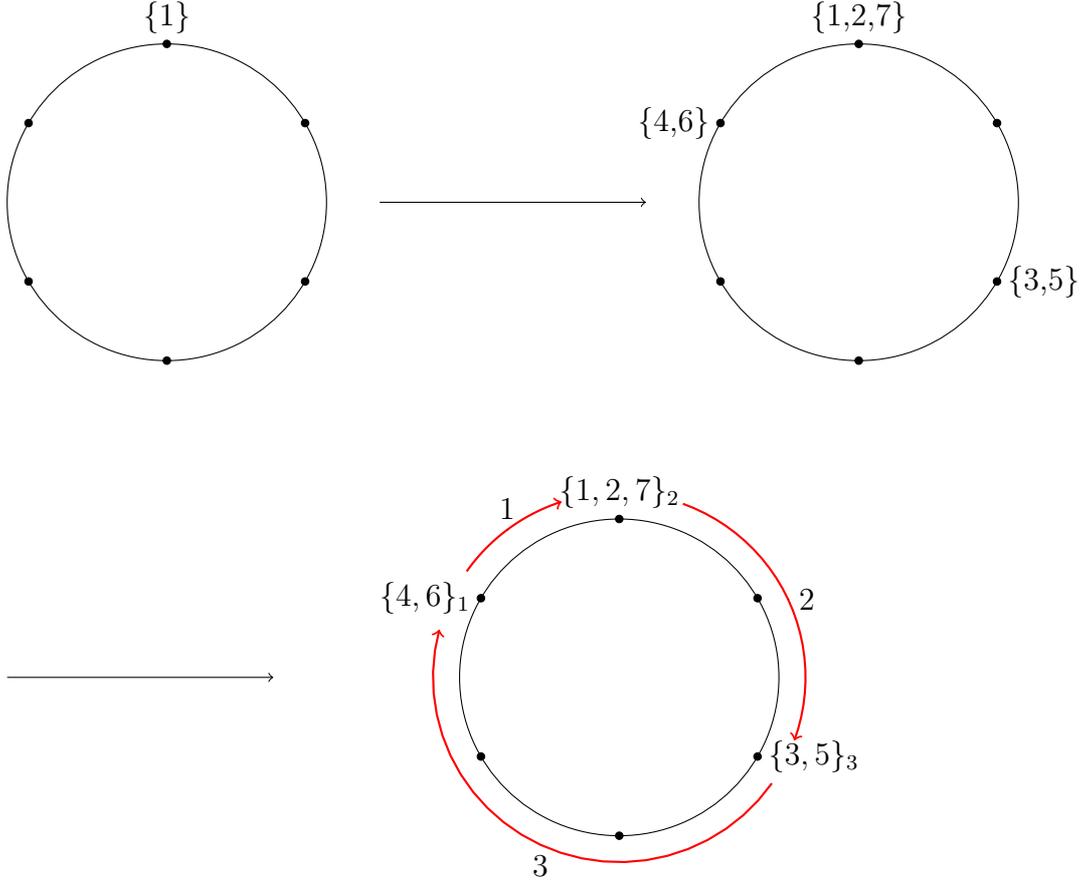
\begin{figure}[ht]\label{f2}
\begin{tikzpicture}[scale=0.7]
\filldraw[black] (0,3) circle (2pt) node[anchor=south] {\{1\}};
\filldraw[black] (2.598,1.5) circle (2pt) node[anchor=west] {};
\filldraw[black] (2.598,-1.5) circle (2pt) node[anchor=west] {};
\filldraw[black] (0,-3) circle (2pt) node[anchor=north] {};
\filldraw[black] (-2.598,-1.5) circle (2pt) node[anchor=east] {};
\filldraw[black] (-2.598,1.5) circle (2pt) node[anchor=east] {};
\draw (0,0) circle (3);
\draw[->] (4,0) -- (9,0);
\filldraw[black] (5+8,3) circle (2pt) node[anchor=south] {\{1,2,7\}};
\filldraw[black] (5+10.598,1.5) circle (2pt) node[anchor=west] {};
\filldraw[black] (5+10.598,-1.5) circle (2pt) node[anchor=west] {\{3,5\}};
\filldraw[black] (5+8,-3) circle (2pt) node[anchor=north] {};
\filldraw[black] (5+8-2.598,-1.5) circle (2pt) node[anchor=east] {};
\filldraw[black] (5+8-2.598,1.5) circle (2pt) node[anchor=east] {\{4,6\}};
\draw (5+8,0) circle (3);
\draw[->] (-3,-9) -- (5-3,-9);

\filldraw[black] (2+6.5,3-9) circle (2pt) node[anchor=south] {$\{1,2,7\}_2$};
\filldraw[black] (2+5+10.598-6.5,1.5-9) circle (2pt) node[anchor=west] {};
\filldraw[black] (2+5+10.598-6.5,-1.5-9) circle (2pt) node[anchor=west] {$\{3,5\}_3$};
\filldraw[black] (2+5+8-6.5,-3-9) circle (2pt) node[anchor=north] {};
\filldraw[black] (2+5+8-2.598-6.5,-1.5-9) circle (2pt) node[anchor=east] {};
\filldraw[black] (2+5+8-2.598-6.5,1.5-9) circle (2pt) node[anchor=east] {$\{4,6\}_1$};
\draw (8.5,-9) circle (3);

\draw [red,thick,<-,domain=-20:70] plot ({8.5+3.5*cos(\x)}, {-9+3.5*sin(\x)});
\draw [red,thick,->,domain=145:108] plot ({8.5+3.5*cos(\x)}, {-9+3.5*sin(\x)});
\draw [red,thick,<-,domain=165:325] plot ({8.5+3.5*cos(\x)}, {-9+3.5*sin(\x)});
\filldraw[black] (11.672,-7.520) circle (0.00001pt) node[anchor=west] {2};
\filldraw[black] (7.0208360839
,-12.172077255
) circle (0.00001pt) node[anchor=north] {3};
\filldraw[black] (6.393647419
,-6.2047757148
) circle (0.00001pt) node[anchor=south] {1};
\end{tikzpicture}
\caption{Constructing the decorated ordered set partition associated to the winding vector $(0,2,3,3,3,1,0)$.} \label{fig2}
\end{figure}

From Proposition \ref{prop2}, we know that the number of decorated ordered set partitions of type $(k,n)$ with winding number $d$ is $|\{(w_1,\cdots,w_n) \in \mathbb{Z}^n \mid 0 \leq w_i \leq k-1 , \hspace{2pt} w_1+\cdots+w_n=kd\}|$. A simple combinatorial argument shows this number is the same as the coefficient of $t^{kd}$ in $(1+\cdots+t^{k-1})^n$, which is $\binom{n}{kd}_k$. So the number of decorated ordered set partitions of type $(k,n)$ with winding number $d$ is $\binom{n}{kd}_{k}$.

Recall that we are interested in the number of $r$-hypersimplicial decorated ordered set partitions of type $(k,n)$ with winding number $d$. Throughout the rest of this section, when we say decorated ordered set partition, \textbf{we always assume it is of type $(k,n)$ with winding number $d$}.

\begin{Def}
For a decorated ordered set partition $P=((L_1)_{l_1},(L_2)_{l_2},...,(L_m)_{l_m})$, a block $L_i$ is \textit{r-bad} if $l_i \geq r|L_i|$. Let $I_r (P)=\{L_i \mid L_i \text{ is $r$-bad}\}$.
\end{Def}

For example, the set $I_1 ((\{1,2,7\}_2,\{3,5\}_3,\{4,6\}_1))$ is $\{\{3,5\}\}$. Recall that  $r$-hypersimplicial decorated ordered set partitions satisfy $1 \leq l_i \leq r|L_i|-1$ for all blocks. So a decorated ordered set partition is $r$-hypersimplicial if and only if $I_r (P)$ is empty.
\begin{Def}
For a set $T$, define $UP(T)$ to be a set of all (unordered) partitions of $T$. For example, the partition $\{\{1,2,4\},\{3\},\{5\}\}$ is in $UP(\{1,2,3,4,5\})$.
\end{Def}

\begin{Def}
For $T\subseteq \{1,2,\cdots,n\}$ and $S\in UP(T)$,
define \\ $K_r (S)=\{ P \text{: decorated ordered set partition such that $S \subseteq I_r (P)$} \}$.
\end{Def}
In other words, the set $K_r (S)$ consists of all decorated ordered set partitions  (of type $(k,n)$ with winding number $d$) having elements of $S$ as $r$-bad blocks. For example, when $S=\phi$, the set $K_r(\phi)$ consists of all decorated ordered set partitions (of type $(k,n)$ with winding number $d$).
\begin{Def}
For $T \subseteq \{1,2,\cdots,n\}$, let $H_r (T)=\sum\limits_{S \in UP(T)} (-1)^{|S|} |K_r (S)|.$
\end{Def}
\begin{Ex}  Table \ref{ta1} shows the lists of  $K_1(S)$ for $S \in UP(\{1,2,3\})$, among decorated ordered set partitions of type $(4,5)$ with winding number 1.
\begin{table}[h!]
\begin{tabu} to 0.8\textwidth { | X[c] | X[c] | }
   \hline
   Set & Elements \\
   \hline
   $K_1 (\{\{1,2,3\}\})$ & $(\{1,2,3\}_3,\{4,5\}_1)$   \\
   \hline
    $K_1 (\{\{1,2\},\{3\}\})$  &  $(\{1,2\}_2,\{3\}_1,\{4,5\}_1)$  \\
   \hline
    $K_1 (\{\{2,3\},\{1\}\})$  & $(\{1\}_1,\{2,3\}_2,\{4,5\}_1)$  \\
   \hline
    $K_1 (\{\{1,3\},\{2\}\})$  &  \\
   \hline
    $K_1 (\{\{1\},\{2\},\{3\}\})$  & $(\{1\}_1,\{2\}_1,\{3\}_1,\{4,5\}_1)$ \\ 
   \hline

\end{tabu}
\vspace{3mm}

\caption{ Listing $K_1(S)$ for $S \in UP(\{1,2,3\})$, among decorated ordered set partitions of type $(4,5)$ with winding number 1.}
\label{ta1}
\end{table}

Note that $K_1 (\{\{1,3\},\{2\}\})$ is an empty set, as it is impossible to have winding number 1 with 1-bad blocks $\{1,3\}$ and $\{2\}$. In this case we have, 
\begin{align*}
H_1 (\{1,2,3\})=-&|K_1 (\{\{1,2,3\}\})|+|K_1 (\{\{1,2\},\{3\}\})|  +|K_1 (\{\{2,3\},\{1\}\})| \\ +&|K_1 (\{\{1,3\},\{2\}\})|-|K_1 (\{\{1\},\{2\},\{3\}\})|  \\=-&1+1+1-1=0.
\end{align*}
\end{Ex}

Now we relate $H_r (T)$ with the number of $r$-hypersimplicial decorated ordered set partitions (of type $(k,n)$ with winding number $d$).

\begin{Prop} \label{prop3} The number of $r$-hypersimplicial decorated ordered set partitions (of type $(k,n)$ with winding number $d$)
is
$$\sum_{T \subseteq \{1,2,...,n\}} H_r (T).$$

\end{Prop}
\begin{proof}
It is enough to compute $\sum\limits_{T \subseteq \{1,2,\cdots,n\}}(\sum\limits_{S \in UP(T)} (-1)^{|S|} |K_r (S)|)$, by the definition of $H_r(T)$. A decorated ordered set partition $P$ belongs to $K_r (S)$ if and only if $S$ is a subset of  $I_r (P)$. So if $I_r (P)$ is empty then $P$ will be counted once when $S=\phi$. If $I_r (P)$ is non empty, say $|I_r (P)|=m$, then $P$ will be counted $\binom{m}{i}$ times with the sign $(-1)^i$ as $S$ ranges over all $i$-element subsets of $I_r (P)$. Thus the contribution of $P$ to ($\sum\limits_{T \subseteq \{1,2,...,n\}} H_r (T)$) is $\sum\limits_{i=0}^{m}(-1)^i \binom{m}{i}=0$. So the above sum counts $P$ such that $I_r (P)$ is empty, which means $r$-hypersimplicial.
\end{proof}

 Now it remains to give a formula for  $H_r(T)$. When $S \in UP(\{1,2,\cdots,n\})$, elements of $K_r (S)$ are decorated ordered set partitions $P=((L_1)_{l_1},\cdots,(L_m)_{l_m})$ whose blocks are all $r$-bad, which means $l_i \geq r |L_i|$ for all $i$. Summing inequalities for all $i$ gives $\sum l_i \geq r \sum |L_i|$ that implies $k \geq rn$ which is impossible as $k<n$. Thus $K_r (S)$ is an empty set, so $H_r (\{1,2,\cdots,n\})=0$. So we will only consider when $T$ is a proper subset of $\{1,2,\cdots,n\}$. By Lemma \ref{lem2}, we may assume that $n \notin T$ since $H_r (T)$ is invariant under cyclic shifts of $\{1,2,\cdots,n\}$.

\begin{Def}
For a fixed $T \subsetneq \{1,2,\cdots,n\}$ such that $n \notin T$, a \textit{T-singlet block} is a block with only one element $t$ and $t\in T$. A sequence of consecutive blocks $(L_i,\cdots,L_{i+j})$ consisting of $T$-singlet blocks in a decorated ordered set partition $P$ (indices are considered modulo number of blocks in $P$) is \textit{r-packed} if $l_i=\cdots=l_{i+j-1}=r$ and $l_{i+j} \geq r$. An $r$-packed sequence is \textit{increasing r-packed} if elements in each block $(L_i,\cdots,L_{i+j})$ are in increasing order. Such a sequence is $maximal$ if it is not a subsequence of another increasing r-packed sequence.  
\end{Def}

The increasing r-packed condition highly depends on $T$ since it only applies to consecutive $T$-singlet blocks. Note that $T$-singlet blocks in $r$-packed sequence are all $r$-bad. It is the most concentrated arrangement that makes these blocks all $r$-bad. We allow increasing $r$-packed sequence of length 1 by convention.

\begin{Ex}
For $T=\{1,2,4,6\}$ and $r=2$, Figure \ref{fig3} is the picture for the decorated ordered set partition $(\{1\}_2,\{2\}_2,\{4\}_2,\{5,8,9,10\}_1,\{6\}_2,\{7\}_2,\{11,12,13\}_1)$. Maximal increasing $r$-packed sequences here are $(\{1\},\{2\},\{4\})$ and $(\{6\})$. Note that the sequence $(\{6\},\{7\})$ is not $r$-packed since $\{7\}$ is not a $T$-singlet block. 
\end{Ex}
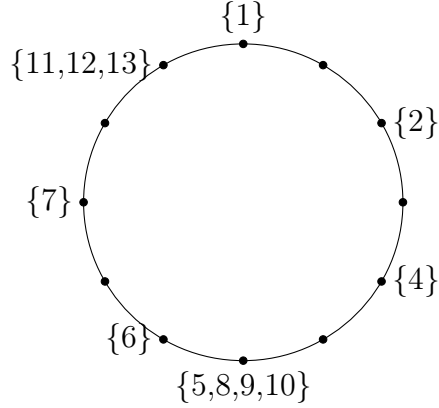
\begin{figure}[ht]\label{f5}
\begin{tikzpicture}[scale=0.7]
\filldraw[black] (0,3) circle (2pt) node[anchor=south] {\{1\}};
\filldraw[black] (1.5,2.5980762114) circle (2pt) node[anchor=south] {};
\filldraw[black] (2.5980762114,1.5) circle (2pt) node[anchor=west] {\{2\}};
\filldraw[black] (3,0) circle (2pt) node[anchor=west]{};
\filldraw[black] (2.5980762114,-1.5) circle (2pt) node[anchor=west] {\{4\}};
\filldraw[black] (1.5,-2.5980762114) circle (2pt) node[anchor=north] {};
\filldraw[black] (0,-3) circle (2pt) node[anchor=north] {\{5,8,9,10\}};
\filldraw[black] (-1.5,-2.5980762114) circle (2pt) node[anchor=east] {\{6\}};
\filldraw[black] (-2.5980762114,-1.5) circle (2pt) node[anchor=west] {};
\filldraw[black] (-3,0) circle (2pt) node[anchor=east] {\{7\}};
\filldraw[black] (-2.5980762114,1.5) circle (2pt) node[anchor=east] {};
\filldraw[black] (-1.5,2.5980762114) circle (2pt) node[anchor=east] {\{11,12,13\}};
\draw (0,0) circle (3);

\end{tikzpicture}
\caption{Reading off $r$-packed sequences for $r=2$.} \label{fig3}
\end{figure}

\begin{Lemma} \label{lem3}
Let $S=\{M_1,M_2,\cdots,M_j\} \in UP(T)$, where $T=\{t_1<t_2<\cdots<t_m\}$ and $n \notin T$. Enumerate the elements of $M_i$ in increasing order, so $M_i=\{t_{i_1}<t_{i_2}<\cdots<t_{i_w}\}$. Then elements of $K_r (S)$ are in bijection with elements of $K_r (\{\{t_1\},\{t_2\},\cdots,\{t_m\}\})$ having increasing $r$-packed sequence $(\{t_{i_1}\},\{t_2\},\cdots,\{t_{i_w}\})$ for all $i$. 
\end{Lemma}

\begin{proof}
Given a decorated ordered set partition $P \in K_r (S)$, we pick a block $(M_i)_l$ which is $r$-bad. So $l \geq r|M_i|=rw$. Change $(M_i)_l$ to $\{t_{i_1}\}_r$,$\{t_{i_2}\}_r$,...,$\{t_{i_w}\}_{l-r(w-1)}$. Since $l-r(w-1) \geq r$, the sequence $(\{t_{i_1}\},\{t_{i_2}\},\cdots,\{t_{i_w}\})$ will be increasing $r$-packed. This process does not change the winding number and new $T$-singlet blocks are all $r$-bad (see Example \ref{ex14}). Repeating this process for all $i$ we get the desired correspondence.
\end{proof}
\begin{Ex}
\label{ex14}
See Figure \ref{fig4}. The figure on the left is a decorated ordered partition $(\{1,2,4\}_6,\{5,8,9,10,13\}_1,\{6,7\}_4,\{11,12\}_1)$. When $T=\{1,2,4,6,7\}$ and $r=2$, the figure on the left has $r$-bad blocks $\{1,2,4\}$ and $\{6,7\}$, so belongs to $K_r (\{\{1,2,4\},\{6,7,\}\})$. Under the correspondence stated in Lemma \ref{lem3}, this goes to $(\{1\}_2,\{2\}_2,\{4\}_2,\{5,8,9,10,13\}_1,\{6,\}_2,\{7\}_2\{11,12\}_1)$, a decorated ordered set partition for the figure on the right. The winding number does not change. 
\end{Ex}

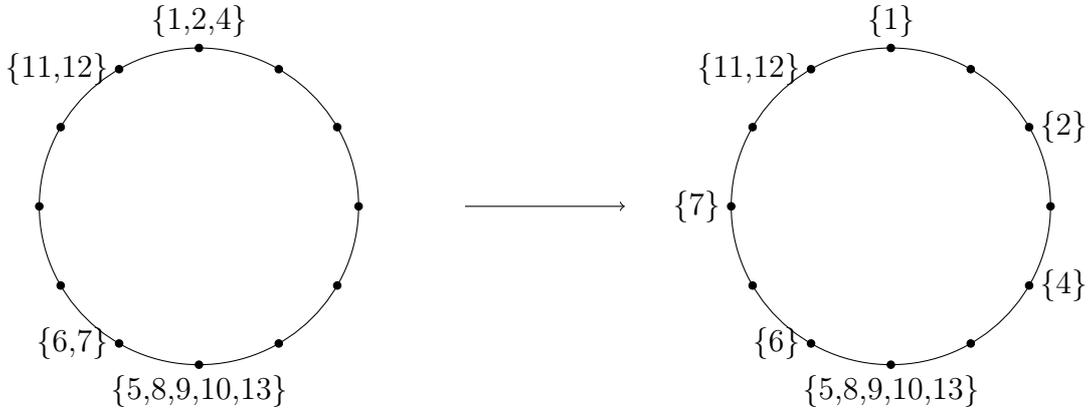
\begin{figure}[ht]
\begin{tikzpicture}[scale=0.7]
\filldraw[black] (0,3) circle (2pt) node[anchor=south] {\{1,2,4\}};
\filldraw[black] (1.5,2.5980762114) circle (2pt) node[anchor=south] {};
\filldraw[black] (2.5980762114,1.5) circle (2pt) node[anchor=south] {};
\filldraw[black] (3,0) circle (2pt) node[anchor=south] {};
\filldraw[black] (2.5980762114,-1.5) circle (2pt) node[anchor=south] {};
\filldraw[black] (1.5,-2.5980762114) circle (2pt) node[anchor=south] {};
\filldraw[black] (0,-3) circle (2pt) node[anchor=north] {\{5,8,9,10,13\}};
\filldraw[black] (-1.5,-2.5980762114) circle (2pt) node[anchor=east]{\{6,7\}} ;
\filldraw[black] (-2.5980762114,-1.5) circle (2pt) node[anchor=south] {};
\filldraw[black] (-3,0) circle (2pt) node[anchor=south] {};
\filldraw[black] (-2.5980762114,1.5) circle (2pt) node[anchor=south] {};
\filldraw[black] (-1.5,2.5980762114) circle (2pt) node[anchor=east] {\{11,12\}};

\draw (0,0) circle (3);
\draw[->] (5,0) -- (8,0);

\filldraw[black] (5+8+0,3) circle (2pt) node[anchor=south] {\{1\}};
\filldraw[black] (13+1.5,2.5980762114) circle (2pt) node[anchor=south] {};
\filldraw[black] (13+2.5980762114,1.5) circle (2pt) node[anchor=west] {\{2\}};
\filldraw[black] (13+3,0) circle (2pt) node[anchor=south] {};
\filldraw[black] (13+2.5980762114,-1.5) circle (2pt) node[anchor=west] {\{4\}};
\filldraw[black] (13+1.5,-2.5980762114) circle (2pt) node[anchor=south] {};
\filldraw[black] (13+0,-3) circle (2pt) node[anchor=north] {\{5,8,9,10,13\}};
\filldraw[black] (13-1.5,-2.5980762114) circle (2pt) node[anchor=east] {\{6\}};
\filldraw[black] (13-2.5980762114,-1.5) circle (2pt) node[anchor=south] {};
\filldraw[black] (13-3,0) circle (2pt) node[anchor=east] {\{7\}};
\filldraw[black] (13-2.5980762114,1.5) circle (2pt) node[anchor=south] {};
\filldraw[black] (13-1.5,2.5980762114) circle (2pt) node[anchor=east] {\{11,12\}};

\draw (5+8,0) circle (3);

\end{tikzpicture}
\caption{Correspondence in Lemma \ref{lem3} for $T=\{1,2,4,6,7\}$ and $r=2$.} \label{fig4}
\end{figure}

\begin{Rem}
 The condition $n\notin T$ is essential for Lemma \ref{lem3}. Without this condition, the correspondence might change the winding number as shown in Figure \ref{fig5}. The winding number on the left figure is 1 but the winding number on the right is 2. We spread elements in blocks in increasing order but since there is a cyclic symmetry, "increasing" might not be meaningful if $n \in T$.   
\end{Rem}
 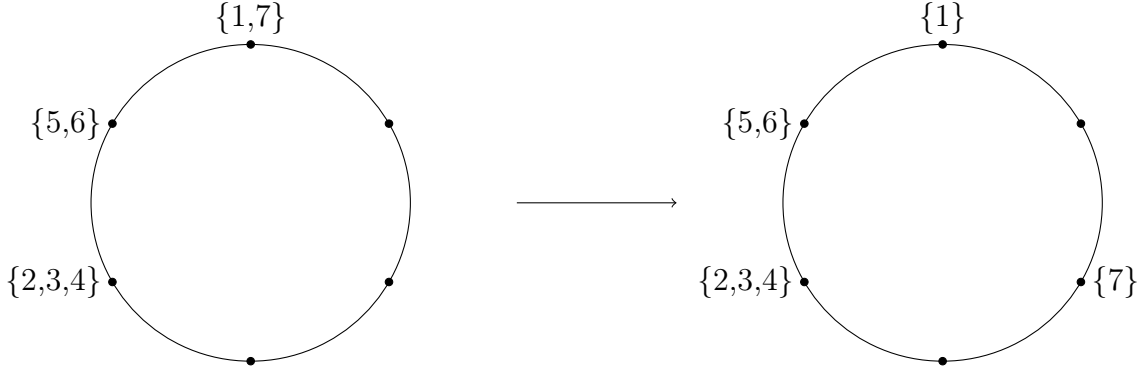
\begin{figure}[ht]
\begin{tikzpicture}[scale=0.7]
\filldraw[black] (0,3) circle (2pt) node[anchor=south] {\{1,7\}};

\filldraw[black] (2.5980762114,1.5) circle (2pt) node[anchor=south] {};
;
\filldraw[black] (2.5980762114,-1.5) circle (2pt) node[anchor=south] {};

\filldraw[black] (0,-3) circle (2pt) node[anchor=north] {};

\filldraw[black] (-2.5980762114,-1.5) circle (2pt) node[anchor=east] {\{2,3,4\}};

\filldraw[black] (-2.5980762114,1.5) circle (2pt) node[anchor=east] {\{5,6\}};

\draw (0,0) circle (3);
\draw[->] (5,0) -- (8,0);

\filldraw[black] (5+8+0,3) circle (2pt) node[anchor=south] {\{1\}};

\filldraw[black] (13+2.5980762114,1.5) circle (2pt) node[anchor=west] {};

\filldraw[black] (13+2.5980762114,-1.5) circle (2pt) node[anchor=west] {\{7\}};

\filldraw[black] (13+0,-3) circle (2pt) node[anchor=north] {};

\filldraw[black] (13-2.5980762114,-1.5) circle (2pt) node[anchor=east] {\{2,3,4\}};

\filldraw[black] (13-2.5980762114,1.5) circle (2pt) node[anchor=east] {\{5,6\}};

\draw (5+8,0) circle (3);

\end{tikzpicture}
\caption{Correspondence in Lemma \ref{lem3} for $T=\{1,7\}$ and $r=2$.} \label{fig5}
\end{figure}
 
 Now fix $T=\{t_1<t_2<\cdots<t_m\} \subsetneq \{1,2,\cdots,n\}$ such that $n \notin T$. For $S \in UP(T)$, the correspondence in Lemma \ref{lem3} gives an embedding $$i_S: K_r(S) \xhookrightarrow{} K_r(\{\{t_1\},\{t_2\},\cdots,\{t_m\}\}).$$
 Let $\chi_S :K_r(\{\{t_1\},\{t_2\},\cdots,\{t_m\}\}) \rightarrow \{0,1\}$ to be the characteristic function of  $i_{S}(K_r(S))$ which means $\chi_S(P)=0$ if $P \notin i_{S}(K_r(S)) $ and $\chi_S(P)=1$ if $P \in i_{S}(K_r(S)) $. Then we have
\begin{align}\label{4}
 H_r(T)=&\sum_{S \in UP(T)} (-1)^{|S|}|K_r(S)|=\sum_{S \in UP(T)} (-1)^{|S|}|i_{S}(K_r(S))| \\ =&\sum_{S\in UP(T)}(-1)^{|S|}(\sum_{P\in K_r(\{\{t_1\},\{t_2\},\cdots,\{t_m\}\})} \chi_S(P)) \nonumber \\ =&\sum_{P\in K_r(\{\{t_1\},\{t_2\},\cdots,\{t_m\}\})}(\sum_{S\in UP(T)}(-1)^{|S|} \chi_S(P)). \nonumber
\end{align}

\begin{Prop}\label{prop4} For a fixed $T=\{t_1<t_2<\cdots<t_m\} \subsetneq \{1,2,\cdots,n\}$ such that $n \notin T$, if a decorated ordered set partition $P \in K_r(\{\{t_1\},\{t_2\},\cdots,\{t_m\}\})$ does not have an increasing $r$-packed sequence of length greater than 1, then $\sum\limits_{S\in UP(T)} (-1)^{|S|} \chi_S (P)$ equals $(-1)^{|T|}$. Otherwise it is zero.  
\end{Prop}
\begin{proof} For $P \in K_r(\{\{t_1\},\{t_2\},\cdots,\{t_m\}\})$, define $\hat{S}(P)$ to be an unordered partition of $T$ by putting $t_i$ and $t_j$ in same part if they belong to same increasing $r$-packed sequence (this will partition $T$ by maximal increasing $r$-packed sequences of $P$). An unordered partition $S$ is a finer partition than $\hat{S}(P)$ if and only if $\chi_S(P)=1$. When $P$ has no increasing $r$-packed sequence of length greater than 1, we have $\hat{S}(P)=\{ \{t_1\},\{t_2\},\cdots,\{t_m\} \}$, the finest unordered partition of $T$. So $\chi_S(P)=1$ only when $S=\hat{S}(P)$ thus $\sum\limits_{S\in UP(T)} (-1)^{|S|} \chi_S (P)=(-1)^{|T|}$. Now assume there is $M=\{m_1<\cdots<m_a\} \in \hat{S}(P)$ such that $|M|=a \geq 2$. To split $M$ into $b$ parts so that resulting finer partition $S$ still satisfies $\chi_S(P)=1$, we choose $(b-1)$ elements $i_1<\cdots<i_{b-1}$ in a set $\{1,\cdots,a-1\}$ and split $M$ into $\{m_1,\cdots,m_{i_1}\},\{m_{i_1},\cdots,m_{i_2}\},\cdots,\{m_{i_{b-1}},\cdots,m_{a}\}$. There are $\binom{a-1}{b-1}$ ways to do that and this process can be done independently on each $M \in \hat{S}(P)$ such that $|M| \geq 2$. So we have
$$\sum\limits_{S\in UP(T)} (-1)^{|S|} \chi_S (P)=\prod_{M \in \hat{S}(P), |M|\geq2}(\sum_{b=1}^{|M|}(-1)^{b} \binom{|M|-1}{b-1})\prod_{M \in \hat{S}(P),|M|=1}(-1).$$
Since $\sum\limits_{b=1}^{|M|}(-1)^{b} \binom{|M|-1}{b-1}=0$, we have $\sum\limits_{S\in UP(T)} (-1)^{|S|} \chi_S (P)=0$ whenever $P$ has an increasing $r$-packed sequence of length greater than 1, that is, the set $\hat{S}(P)$ has a part with more than one element. 

\end{proof}

\begin{Ex}
For $T=\{1,2,3,4\}$, assume $P \in K_r (\{\{1\},\{2\},\{3\},\{4\}\})$ has (maximal) increasing $r$-packed sequence $(\{1\},\{2\},\{3\},\{4\})$. We will list $S\in UP(T)$ such that $\chi_S(P)=1$ by number of elements.

$|S|=1$ $\rightarrow$ $\{\{1,2,3,4\}\}$ 

$|S|=2$ $\rightarrow$ $\{\{1\},\{2,3,4\}\}$,$\{\{1,2\},\{3,4\}\}$,$\{\{1,2,3\},\{4\}\}$

$|S|=3$ $\rightarrow$ $\{\{1\},\{2\},\{3,4\}\}$,$\{\{1\},\{2,3\},\{4\}\}$,$\{\{1,2\},\{3\},\{4\}\}$

$|S|=4$ $\rightarrow$ $\{\{1\},\{2\},\{3\},\{4\}\}$ \\
So we have $\sum\limits_{S\in UP(T)} -(1)^{|S|} \chi_S(P)=-1+3-3+1=-\binom{3}{0}+\binom{3}{1}-\binom{3}{2}+\binom{3}{3}=0$.

\end{Ex}
\begin{Def} For a fixed $T=\{t_1<t_2<\cdots<t_m\} \subsetneq \{1,2,\cdots,n\}$ such that $n \notin T$, define $\hat{K_r}(T)$ to be the subset of $K_r (\{\{t_1\},\{t_2\},\cdots,\{t_m\}\})$ consisting of decorated ordered set partitions without increasing $r$-packed sequence of length greater than 1.
\end{Def}
By  Proposition \ref{prop4} and (\ref{4}), we have \begin{equation}H_r(T)=(-1)^{|T|}|\hat{K_r}(T)|.\end{equation} We will count the number of elements in $\hat{K_r}(T)$ by defining the second winding vector for each element. The second winding vector is a modified version of the winding vector that we previously defined.  
 
 Assume we are given $P \in \hat{K}_r(T)$. There are $k$ spots total on the circle including empty spots that are recording distances and $T$-singlet blocks $\{t_1\},\{t_2\},\cdots,\{t_m\}$ are $r$-bad blocks so for each $\{t_i\}$, there will be at least $(r-1)$ empty spots after $\{t_i\}$ as the distance to the next block is at least $r$. Color these $r$ spots, that is, the spot occupied by $\{t_i\}$ with $(r-1)$ empty spots after that \textbf{red}. Doing this for all $i$, total $r|T|=rm$ spots will be colored red. And color the remaining $(k-rm)$ spots \textbf{blue}. 
 
 \begin{Def} For $P \in \hat{K}_r(T)$, \textit{second winding vector} $v=(v_1,v_2,\cdots,v_n)$ is defined by setting $v_i$ to be the number of \textbf{blue spots} passed while moving from $i$ to $(i+1)$ in clockwise fashion. Do not include the starting point but include the arriving point (if it's blue) and when the starting point and the arriving point are in same block (spot), set $v_i=0$.
  \end{Def}
 Since the winding number is $d$, the whole path winds around the circle $d$ times. So we have $v_1+\cdots+v_n=(k-rm)d$. 
 
 If $i \notin T$, we are starting from the blue spot so $v_i$ can range from 0 to $(k-rm-1)$. However when $i\in T$, we claim $v_i$ cannot be zero. If $v_i=0$, then the path from $i$ to $i+1$ should not include any blue spots. So the path will be of the form $\{i\},\phi,\cdots,\phi,\{a_1\},\phi,\cdots,\phi,\cdots,\{a_q\},\phi,\cdots,\phi,\{i+1\}$ where $\phi$ means an empty spot. Thus the sequence $(\{i\},\{a_1\},\cdots,\{a_q\},\{i+1\})$ is \textit{r-packed}. Since $P$ does not have increasing $r$-packed sequence of length greater than 1, the sequence $(i,a_1,\cdots,a_2,i+1)$ should be a decreasing sequence which is impossible.
  It is possible to have $v_i=k-rm$ as the path can encounter every blue spot (see Example \ref{ex8}). We conclude $1 \leq v_i \leq k-rm$.

\begin{Ex} \label{ex8}
Figure \ref{fig6} explains the way to read off second winding vector. Let $P=(\{2\}_2,\{1\}_2,\{5,6\}_1,\{7,8\}_1,\{9\}_3,\{11,12,13\}_1,\{10,14\}_1,\{3,4\}_1)$, and fix $T=\{1,2,9\}$ and $r=2$. The upper left figure is a picture for $P$. Note that the sequence $(\{2\},\{1\})$ is $r$-packed but not increasing $r$-packed. So $P$ has no increasing $r$-packed sequence of length greater than 1. After coloring spots with the rule above we get the upper right figure. There will be $r|T|$ (=6) red spots and $(k-r|T|)$ (=6) blue spots. To get $v_1$, wind from 1 to 2 clockwise as shown in the lower figure, and count the number of blue spots passed. Here $v_1=6$. Continuing this process we have the second winding vector $v=(6,6,0,1,0,1,0,0,3,5,0,0,1,1)$.
\end{Ex}

\begin{figure}[ht]
\begin{tikzpicture}[scale=0.65]
\filldraw[black] (0,3) circle (2pt) node[anchor=south] {\{2\}};
\filldraw[black] (1.5,2.5980762114) circle (2pt) node[anchor=south] {};
\filldraw[black] (2.5980762114,1.5) circle (2pt) node[anchor=west] {\{1\}};
\filldraw[black] (3,0) circle (2pt) node[anchor=west] {};
\filldraw[black] (2.5980762114,-1.5) circle (2pt) node[anchor=west] {\{5,6\}};
\filldraw[black] (1.5,-2.5980762114) circle (2pt) node[anchor=north] {\{7,8\}};
\filldraw[black] (0,-3) circle (2pt) node[anchor=north] {\{9\}};
\filldraw[black] (-1.5,-2.5980762114) circle (2pt) node[anchor=north] {};
\filldraw[black] (-2.5980762114,-1.5) circle (2pt) node[anchor=west] {};
\filldraw[black] (-3,0) circle (2pt) node[anchor=east] {\{11,12,13\}};
\filldraw[black] (-2.5980762114,1.5) circle (2pt) node[anchor=east] {\{10,14\}};
\filldraw[black] (-1.5,2.5980762114) circle (2pt) node[anchor=south] {\{3,4\}};
\draw (0,0) circle (3);

\draw (12.5,0) circle (3);
\filldraw[red] (12.5+0,3) circle (2pt) node[anchor=south] {\{2\}};
\filldraw[red] (12.5+1.5,2.5980762114) circle (2pt) node[anchor=south] {};
\filldraw[red] (12.5+2.5980762114,1.5) circle (2pt) node[anchor=west] {\{1\}};
\filldraw[red] (12.5+3,0) circle (2pt) node[anchor=west] {};
\filldraw[blue] (12.5+2.5980762114,-1.5) circle (2pt) node[anchor=west] {\{5,6\}};
\filldraw[blue] (12.5+1.5,-2.5980762114) circle (2pt) node[anchor=north] {\{7,8\}};
\filldraw[red] (12.5+0,-3) circle (2pt) node[anchor=north] {\{9\}};
\filldraw[red] (12.5-1.5,-2.5980762114) circle (2pt) node[anchor=north] {};
\filldraw[blue] (12.5-2.5980762114,-1.5) circle (2pt) node[anchor=west] {};
\filldraw[blue] (12.5-3,0) circle (2pt) node[anchor=east] {\{11,12,13\}};
\filldraw[blue] (12.5-2.5980762114,1.5) circle (2pt) node[anchor=east] {\{10,14\}};
\filldraw[blue] (12.5-1.5,2.5980762114) circle (2pt) node[anchor=south] {\{3,4\}};

\draw[->] (4.3,2) -- (7.2,2);

\draw[->] (-3.5,-9) -- (1.5,-9);

\filldraw[red] (8.5+0,3-9) circle (2pt) node[anchor=south] {\{2\}};
\filldraw[red] (8.5+1.5,2.5980762114-9) circle (2pt) node[anchor=south] {};
\filldraw[red] (8.5+2.5980762114,1.5-9) circle (2pt) node[anchor=west] {\{1\}};
\filldraw[red] (8.5+3,0-9) circle (2pt) node[anchor=west] {};
\filldraw[blue] (8.5+2.5980762114,-1.5-9) circle (2pt) node[anchor=west] {\{5,6\}};
\filldraw[blue] (8.5+1.5,-2.5980762114-9) circle (2pt) node[anchor=north] {\{7,8\}};
\filldraw[red] (8.5+0,-3-9) circle (2pt) node[anchor=north] {\{9\}};
\filldraw[red] (8.5-1.5,-2.5980762114-9) circle (2pt) node[anchor=north] {};
\filldraw[blue] (8.5-2.5980762114,-1.5-9) circle (2pt) node[anchor=west] {};
\filldraw[blue] (8.5-3,0-9) circle (2pt) node[anchor=east] {\{11,12,13\}};
\filldraw[blue] (8.5-2.5980762114,1.5-9) circle (2pt) node[anchor=east] {\{10,14\}};
\filldraw[blue] (8.5-1.5,2.5980762114-9) circle (2pt) node[anchor=south] {\{3,4\}};
\draw (8.5,-9) circle (3);

\draw [green,thick,<-,domain=-260:20] plot ({8.5+3.5*cos(\x)}, {-9+3.5*sin(\x)});

\end{tikzpicture}
\caption{Reading off the second winding vector.} \label{fig6}
\end{figure}
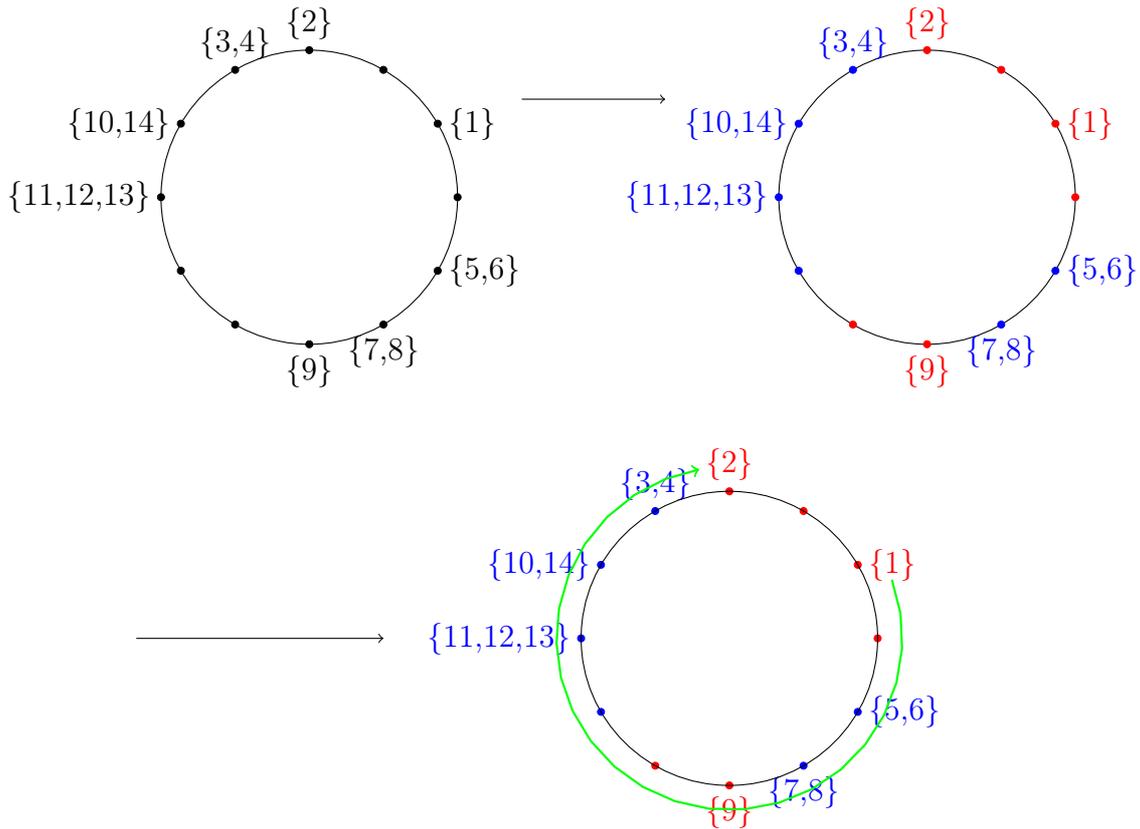
We saw that a second winding vector $v=(v_1,v_2,\cdots,v_n)$ satisfies $v_1+\cdots+v_n=(k-rm)d$. And it also satisfies $0 \leq v_i \leq k-rm-1$ if $i \notin T$, and $1 \leq v_i \leq k-rm$ if $i \in T.$ 

It turns out these are the only restrictions for the second winding vectors of the elements of $\hat{K}_r(T)$.

\begin{Prop} \label{prop5}
Elements of $\hat{K}_r(T)$, where $|T|=m$, are in bijection with elements of
$\{(v_1,v_2,\cdots,v_n) \in \mathbb{Z}^n \mid 0 \leq v_i \leq k-rm-1 \hspace{4pt} \text{if} \hspace{4pt} i \notin T, 1 \leq v_i \leq k-rm \hspace{4pt} \text{if} \hspace{4pt} i \in T, \hspace{4pt} v_1+\cdots+v_n=(k-rm)d\}$.
\end{Prop}
\begin{proof}
The forward direction is done by the second winding vector. For the reverse direction, we should recover the decorated ordered set partition (in $\hat{K}_r(T)$) whose second winding vector is the specified vector $(v_1,v_2,\cdots,v_n)$. First draw $(k-rm)$ spots on the circle (recall $|T|=m$) and put 1 in one spot. Having put $i$ in some spot, move clockwise $w_i$ spots and put $i+1$ in that spot. After placing every element, let's denote the resulting decorated ordered set partition with $P$. We construct $\tilde{P} \in \hat{K}_r(T)$ as follows. For each block $B$ of $P$ with $B \cap T \neq \phi$, let $B \cap T=\{i_1 < \cdots <i_s\}$. We replace $B$ with $B \backslash T$ and then add $(rs)$ spots immediately after $B \backslash T$ as follows: first a $T$-singlet block $\{i_s\}$ then $(r-1)$ empty spots then $T$-singlet block $\{i_{s-1}\}$ then  $(r-1)$ empty spots $\cdots$ $T$-singlet block $\{i_{1}\}$ then  $(r-1)$ empty spots. Resulting decorated ordered set partition $\tilde{P}$ belongs to  $\hat{K}_r(T)$ as all element in $T$ are in $T$-singlet blocks and there is no increasing $r$-packed sequence of length greater than 1 since we placed $i_1 < \cdots <i_s$ in a decreasing order.

It remains to prove the second winding vector $(\tilde{v_1},\tilde{v_2},\cdots,\tilde{v_n})$ of $\tilde{P}$ is the given vector $(v_1,v_2,\cdots,v_n)$. If $i$ and $(i+1)$ were in different blocks in $P$, then $v_i=\tilde{v_i}$ as we ignore red spots on the way. If $i$ and $(i+1)$ were in a same block in $P$ and $i \notin T$, then $v_i=0$. From the construction of $\tilde{P}$, there is no blue spot on the way from $i$ to $(i+1)$ except the starting spot so $\tilde{v_i}=0$.  If $i$ and $(i+1)$ were in a same block in $P$ and $i \in T$, then $v_i=k-rm$. Since $(i+1)$ is located behind $i$, to get from $i$ to $(i+1)$ in $\tilde{P}$ the path winds the circle and encounters every blue spot. Thus $\tilde{v_i}=k-rm$. We conclude that the second winding vector of $\tilde{P}$ is the given vector.  
\end{proof}
\begin{Ex}
Figure \ref{fig7} shows how to recover a decorated ordered set partition from a second winding vector as stated in Proposition \ref{prop5}. We are given $T=\{1,2,9\}$, the number $r=2$, and the second winding vector $v=(6,6,0,1,0,1,0,0,3,5,0,0,1,1)$. In the upper left figure, there are $6=k-r|T|$ spots ($k=12$) on the circle and 1 is in one spot. Then put elements according to the second winding vector. The upper right figure shows this. The elements in $T$ are denoted with a tilde. Consider the first block $\{\tilde{1},\tilde{2},3,4\}$. The numbers 3 and 4 will form a block and 1 and 2 will spread to the right into the space between blocks $\{\tilde{1},\tilde{2},3,4\}$ and $\{5,6\}$, making four new red spots. The same thing happens for the block $\{7,8,\tilde{9}\}$, making two new red spots. The lower figure is the picture for the resulting decorated ordered set partition in $\hat{K}_r(T)$. We recovered Example \ref{ex8}.
\end{Ex}
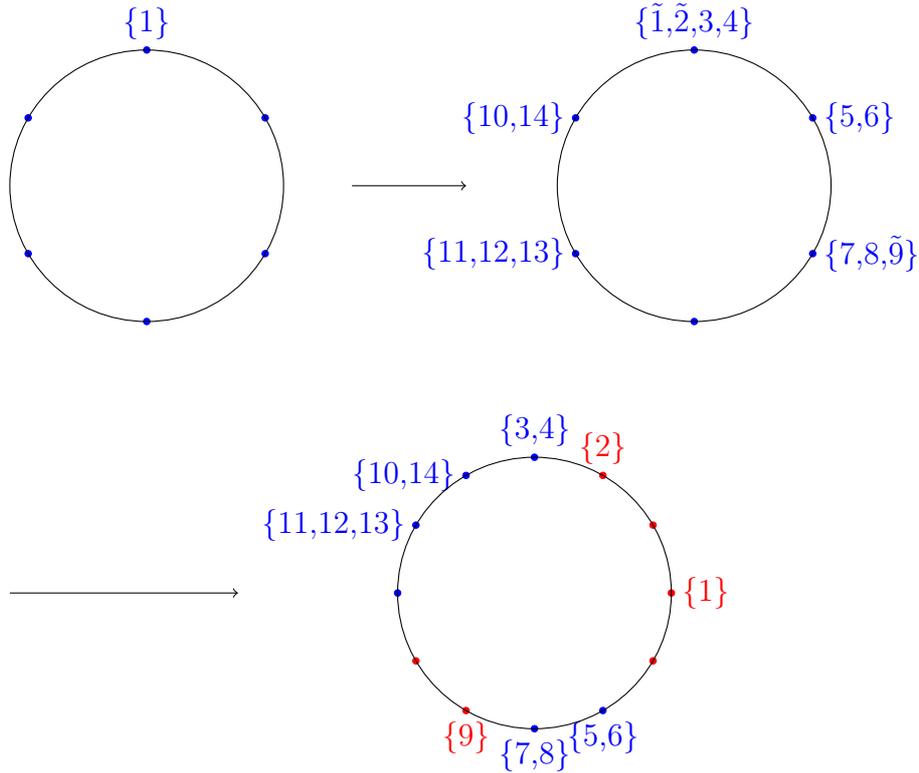
\begin{figure}[ht]
\begin{tikzpicture}[scale=0.6]
\filldraw[blue] (0,3) circle (2pt) node[anchor=south] {\{1\}};
\filldraw[blue] (2.5980762114,1.5) circle (2pt) node[anchor=west] {};
\filldraw[blue] (2.5980762114,-1.5) circle (2pt) node[anchor=west] {};
\filldraw[blue] (0,-3) circle (2pt) node[anchor=west] {};
\filldraw[blue] (-2.5980762114,-1.5) circle (2pt) node[anchor=north] {};
\filldraw[blue] (-2.5980762114,1.5) circle (2pt) node[anchor=east] {};

\draw (0,0) circle (3);

\draw[->] (4.5,0) -- (7,0);

\filldraw[blue] (12,3) circle (2pt) node[anchor=south] {\{$\tilde{1}$,$\tilde{2}$,3,4\}};

\filldraw[blue] (12+2.5980762114,1.5) circle (2pt) node[anchor=west]{\{5,6\}};
\filldraw[blue] (12+2.5980762114,-1.5) circle (2pt) node[anchor=west] {\{7,8,$\tilde{9}$\}};
\filldraw[blue] (12,-3) circle (2pt) node[anchor=north] {};
\filldraw[blue] (12-2.5980762114,-1.5) circle (2pt) node[anchor=east]{\{11,12,13\}};
\filldraw[blue] (12-2.5980762114,1.5) circle (2pt) node[anchor=east] {\{10,14\}};

\draw (12,0) circle (3);

\draw[->] (-3,-9) -- (2,-9);

\filldraw[blue] (8.5+0,3-9) circle (2pt) node[anchor=south] {\{3,4\}};
\filldraw[red] (8.5+1.5,2.5980762114-9) circle (2pt) node[anchor=south] {\{2\}};
\filldraw[red] (8.5+2.5980762114,1.5-9) circle (2pt) node[anchor=west] {};
\filldraw[red] (8.5+3,0-9) circle (2pt) node[anchor=west] {\{1\}};
\filldraw[red] (8.5+2.5980762114,-1.5-9) circle (2pt) node[anchor=west] {};
\filldraw[blue] (8.5+1.5,-2.5980762114-9) circle (2pt) node[anchor=north] {\{5,6\}};
\filldraw[blue] (8.5+0,-3-9) circle (2pt) node[anchor=north] {\{7,8\}};
\filldraw[red] (8.5-1.5,-2.5980762114-9) circle (2pt) node[anchor=north] {\{9\}};
\filldraw[red] (8.5-2.5980762114,-1.5-9) circle (2pt) node[anchor=west] {};
\filldraw[blue] (8.5-3,0-9) circle (2pt) node[anchor=east] {};
\filldraw[blue] (8.5-2.5980762114,1.5-9) circle (2pt) node[anchor=east] {\{11,12,13\}};
\filldraw[blue] (8.5-1.5,2.5980762114-9) circle (2pt) node[anchor=east] {\{10,14\}};
\draw (8.5,-9) circle (3);

\end{tikzpicture}
\caption{Constructing the decorated ordered set partition associated to the second winding vector $v=(6,6,0,1,0,1,0,0,3,5,0,0,1,1)$.} \label{fig7}
\end{figure}

For a second winding vector $v=(v_1,\cdots,v_n)$, let $v'=(v'_1,\cdots,v'_n)$ be a vector such that $v'_i=v_i$ if $i \notin T$, and $v'_i=v_i-1$ if $i \in T$. By the property of a second winding vector, we have $0 \leq v'_i \leq (k-rm-1)$ and $v'_1+\cdots+v'_n=(k-rm)d-|T|=(k-rm)d-m$.
So the number of such $v'$ is $\binom{n}{(k-rm)d-m}_{k-rm}$ which gives

\begin{equation} \label{6}
H_{r}(T)=(-1)^{|T|}|\hat{K}_r (T)|=(-1)^{m}\binom{n}{(k-rm)d-m}_{k-rm}.
\end{equation}

\vspace{10pt}

\textit{Proof of Conjecture \ref{conj2})} By Proposition \ref{prop3}, and the equation (\ref{6}), the number of $r$-hypersimplicial decorated ordered set partitions (of type $(k,n)$ with winding number $d$) is 
 $$\sum_{T \subseteq \{1,2,...,n\}} H_r (T)=\sum_{m\geq0}\left(\sum_{|T|=m} H_r (T)\right)=\sum_{m\geq0} (-1)^{m} \binom{n}{m} \binom{n}{(k-rm)d-m}_{k-rm}.$$
 
 Now comparing with the formula (\ref{5}), we obtain  $Conjecture$ \ref{conj2}. By specializing to $r=1$ we obtain $Conjecture$ \ref{conj1}.\qed

 \vspace{5mm}
 \textbf{Acknowledgments:} The author would like to thank Lauren Williams for pointing out this problem and her helpful comments on drafts of this paper, and Melissa Sherman-Bennet for helping me revise this paper. The author is also grateful to Nick Early for helpful explanations about the background of this conjecture.


\begin{thebibliography}{99} 

\bibitem{Katz} M. Katzman, The Hilbert series of algebras of Veronese type, Comm. Algebra 33 (2005), 1141-1146.

\bibitem{Li} N. Li, Ehrhart h-vectors of hypersimplices, Discrete Comput. Geom. 48 (2012), 847-878.

\bibitem{Early1} N. Early, Conjectures for Ehrhart h-vectors of hypersimplices and dilated simplices, arXiv:1710.09507.

\bibitem{Sta}  R. Stanley, Decompositions of rational convex polytopes, Annals of Discrete Math, volume 6 (1980), 333-342.

\bibitem{Ocn} A. Ocneanu. On the inner structure of a permutation: Bicolored Partitions and Eulerians, Trees and Primitives. arXiv preprint arXiv:1304.1263 (2013). 
 
\bibitem{Early2}  N. Early. Combinatorics and Representation Theory for Generalized Permutohedra I: Simplicial Plates. arXiv preprint arXiv:1611.06640 (2016).

\bibitem{EC1}  R. Stanley, Enumerative Combinatorics, Volume 1, Cambridge Studies in Advanced Mathematics, (2011), p 566-572.

\bibitem{Stan}  R. Stanley, Eulerian Partitions of a Unit Hypercube, Proceedings of the NATO Advanced Study Institute, (1977).


\end{thebibliography}
\end{document}